\documentclass[10pt]{article}
\usepackage{hyperref}
\usepackage{latexsym,amsfonts,amssymb,amsmath,amsthm,graphicx}

        \newtheorem{theorem}{Theorem}[section]

        \newtheorem{proposition}{Proposition}[section]

        \newtheorem{lemma}{Lemma}[section]

\begin{document}

\title{The Minimal Non-Koszul A($\Gamma$)}
\author{David Nacin \\ William Paterson University \\ Wayne, NJ 07470}

\maketitle

\begin{abstract}

The algebras $A(\Gamma)$, where $\Gamma$ is a directed layered
graph, were first constructed by I. Gelfand, S. Serconek, V. Retakh
and R. Wilson.  These algebras are generalizations of the algebras
$Q_n$, which are related to factorizations of non-commutative
polynomials. It was conjectured that these algebras were Koszul. In
2008, T.Cassidy and B.Shelton found a counterexample to this claim,
a non-Koszul $A(\Gamma)$ corresponding to a graph $\Gamma$ with 18
edges and 11 vertices.

We produce an example of a directed layered graph $\Gamma$ with 13
edges and 9 vertices which produces a non-Koszul $A(\Gamma)$.  We
also show this is the minimal example with this property.

\end{abstract}

\section{Introduction}

The relationship between the factorizations and coefficients of a
non-commutative polynomial is described in terms of pseudo-roots by
the non-commutative version of Vieta's theorem \cite{quasi}.  The
algebra $Q_n$ of the pseudo-roots for a polynomial of degree $n$ has
been described and studied by Gelfand, Retakh and Wilson in
\cite{1stqn}. These algebras are quadratic and Koszul, and their
corresponding dual algebras $Q_n^!$ have finite dimension.

The generators of this algebra correspond to elements of $B_n$, the
boolean lattice of all subsets of an $n-$element set.  This
construction was then generalized to form the class of algebras
$A(\Gamma)$, each of which is determined by a layered graph
$\Gamma$. The algebras $Q_n$ are simply the algebras $A(\Gamma)$,
where $\Gamma$ equals $B_n$.

Depending on a condition called uniformity of the graph $\Gamma$,
the algebra $A(\Gamma)$ may be quadratic \cite{firstuniform}, thus
leading to the question of which of these algebras are Koszul.  It
was discovered that the algebras were Koszul for the boolean lattice
$B_n$, simplicial complexes and complete layered graphs with
arbitrary numbers of vertices at each level \cite{manyshownkoszul,
justWS}. It was also conjectured that all algebras $A(\Gamma)$ were
Koszul. This was shown not to be the case when Cassidy and Shelton
found the first example of a non-Koszul $A(\Gamma)$ \cite{cs}; see
Figure \ref{CS}.

\begin{figure}[htb!]
\centering%
\includegraphics[height=30mm]{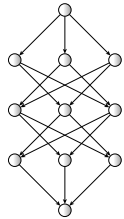}
\caption{A Non-Koszul $A(\Gamma)$ with Eleven Vertices} \label{CS}
\end{figure}

This led to the question of what the smallest $\Gamma$ with a
quadratic non-Koszul $A(\Gamma)$ might be. In 2010, Retakh,
Serconek, and Wilson found that when one edge is removed from a
particular layer, we retain both uniformity and non-Koszulity of
$A(\Gamma)$ \cite{cohom} (see Figure \ref{CSminus}.)

\begin{figure}[h!]
\centering%
\includegraphics[height=30mm]{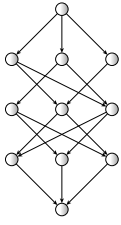}
\caption{Another Non-Koszul $A(\Gamma)$ with Eleven Vertices}
\label{CSminus}
\end{figure}

By computer, we found that we could extend the example further,
removing one vertex from the second highest layer though this was
not minimal; see Figure \ref{CSminusvert}.
\begin{figure}[h!]
\centering%
\includegraphics[height=30mm]{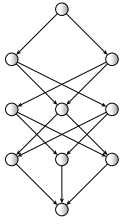}
\caption{A Non-Koszul $A(\Gamma)$ with Ten Vertices}
\label{CSminusvert}
\end{figure}
In this paper we now produce the smallest non-Koszul example.  We
named this graph $H$ after HiLGA,\footnote{The name HiLGA comes from
``Hilbert Series of Layered Graph Algebras''.} the program we wrote
to study algebras of the form $A(\Gamma).$  Computers are not needed
for showing either the minimality in terms of number of vertices, or
the non-Koszulity of $H.$

\section{Theory}

Let $\Gamma=(V,E)$ be a directed graph, where $V = \coprod_{i=0}^{N}
V_i$ and $V_i \neq \emptyset$ for $i \in \{0, 1, \cdots, N\}.$  We
call these $V_i$ the \textit{layers} of $\Gamma$. For each vertex
$v$, we write $|v|=i$, and say that $v$ has height $i$ if $v \in
V_i$. For each edge, let $t(e)$ denote the tail of $e$ and $h(e)$
denote the head of $e$. We say that $\Gamma$ is a \textit{layered
graph of height $N$} if $|t(e)| = |h(e)| + 1$ for all $e \in E.$

Given a layered graph, we construct an algebra $A(\Gamma)$ as
follows:  Begin with the free algebra $T(E)$ generated by the edges
of $\Gamma$.  Let $\pi$ and $\pi'$ be any two paths $\pi = (e_1,
e_2, \cdots e_k)$ and $\pi' = (f_1, f_2, \cdots f_k)$ with the same
starting and finishing vertices. For every such pair $\pi, \pi'$, we
impose the relation
$$(t-e_1)(t-e_2) \cdots (t-e_k) = (t-f_1)(t-f_2) \cdots (t-f_k)$$
where $t$ is a formal parameter commuting with all edges in $E.$
Matching the coefficients of $t$ gives us a collection of relations.
Let $I$ be the collection of all of these relations for any two
paths meeting our conditions.  Then $A(\Gamma)$ is $T(E)/I.$

The standard definition of Koszulity requires a quadratic algebra.
Not all algebras of the form $A(\Gamma)$ are quadratic, but there is
a condition on $\Gamma$ which guarantees $A(\Gamma)$ is quadratic.
Two vertices $v,v'$ of the same layer $l$ are \textit{connected by
an up-down sequence} if there is a sequence $v=v_0, v_1, \cdots,
v_k=v'$ of vertices of level $l$ with the following property: For
each $i$ in $\{1, \cdots, k\}$ there are edges $e$ and $f$ so
$t(e)=v_{i-1}$, $t(f)=v_i$ and $h(e)=h(f).$  A layered graph
$\Gamma$ is said to be \textit{uniform} if for any pair of edges
$e,e'$ with a common tail, their heads are connected by an up down
sequence $v_0, v_1, \cdots v_k$ with $v_i$ adjacent to $t(e)$ for $i
\in \{1, \cdots, k\}.$

It has been shown in \cite{firstuniform} that if $\Gamma$ is a
uniform layered graph then $A(\Gamma)$ is a quadratic algebra. If
$\Gamma$ has a unique minimal vertex of layer zero then there is a
construction giving us a new algebra we call $B(\Gamma).$ The
algebras $B(\Gamma)$ can be presented by generators $u \in V^+ =
\coprod_{i=1}^{N} V_i$ and relations

\begin{enumerate}
\item $u \cdot w = 0$ if there is no edge from $u$ to $w$

\item $u \cdot \sum_{w \in S(u)} w = 0$  where $S(u)$ is the collection of vertices $w$ for which there is an
edge from $u$ to $w$.
\end{enumerate}
The algebras $B(\Gamma)$ are the dual algebras of the associated
graded algebras of $A(\Gamma)$ under a certain filtration.  They
have the same Hilbert series as the actual dual of $A(\Gamma)$ and
when it exists, $B(\Gamma)$ is Koszul if and only if $A(\Gamma)$ is
\cite{cohom}.

\section{A Layered Graph $\Gamma$ on Nine Vertices with non-Koszul A($\Gamma$)}

We introduce a layered graph $H$ on nine vertices; see Figure
\ref{Hv2}.

\begin{figure}[htb!]
\centering%
\includegraphics[width=20mm]{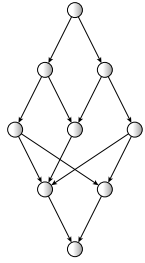}
\caption{The Poset $H$} \label{Hv2}
\end{figure}

\begin{theorem} \label{H}

The algebra $A(\Gamma)$ is not Koszul for $\Gamma = H$.

\end{theorem}

\begin{proof}

We use the numerical Koszulity test from theorem 4.2.1 in
\cite{cohom}. Considering the $i=4$ case we know that $A(\Gamma)$ is
not numerically Koszul if $\dim(H^{-1}(\Gamma_{a,4})) -
\dim(H^{0}(\Gamma_{a,4})) + \dim(H^{1}(\Gamma_{a,4})) \neq 0$ where
the layer of $a$ is greater than or equal to 4.  Thus $a$ must be
the maximal vertex in $H$ so $\Gamma_{a,4}$ is the graph in Figure
\ref{Hgam}.

\begin{figure}[h!]
\centering%
\includegraphics[width=20mm]{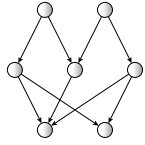}
\caption{The Graph $\Gamma_{a,4}$} \label{Hgam}
\end{figure}

By the Euler-Pointcare formula, we know $\sum (-1)^i \dim(H^i) =
\sum (-1)^j \dim(C^j).$  Thus we have

$$\dim(H^{-1}(\Gamma_{a,4})) - \dim(H^{0}(\Gamma_{a,4})) +
\dim(H^{1}(\Gamma_{a,4})) =$$

$$\dim(C^{-1}(\Gamma_{a,4})) - \dim(C^{0}(\Gamma_{a,4})) +
\dim(C^{1}(\Gamma_{a,4})) - \dim(C^{2}(\Gamma_{a,4})) +
\dim(H^{2}(\Gamma_{a,4})) \geq $$

$$\dim(C^{-1}(\Gamma_{a,4})) - \dim(C^{0}(\Gamma_{a,4})) +
\dim(C^{1}(\Gamma_{a,4})) - \dim(C^{2}(\Gamma_{a,4}))$$

$$= 1 - 7 + 13 - 6 = 1.$$

This shows that $H$ is not numerically Koszul.  Since Koszulity
implies numerical Koszulity, $H$ is not Koszul as well.

\end{proof}

\section{The Minimal non-Koszul A($\Gamma$)}

We now wish to show that $H$ is the unique minimal example.  We must
restrict ourselves to uniform $\Gamma$ in order to guarantee that
the algebra is quadratic \footnote{There is still room, however, to
generalize here using one of the different definitions of Koszulity
for non-quadratic algebras.}.

We will need part of a proposition of Polishchuk and Positselksi
found in \cite{QA} which we summarize here for clarity.

\begin{proposition} \label{PPP}

Let $W$ be a vector space and $X_1, \cdots, X_N \subset W$ be a
collection of its subspaces.  Then the following are equivalent:

\begin{enumerate}

\item The collection $X_1, \cdots, X_N$ is distributive.

\item There exists a direct sum decomposition $W = \bigoplus_{\eta \in
H} W_\eta$ of the vector space $W$ such that each of the subspaces
$X_i$ is the sum of a set of subspaces $W_\eta.$

\end{enumerate}

\end{proposition}

\begin{lemma} \label{YZ}

Suppose the vector subspaces $W_1, W_2, \cdots, W_M$ and $X_1, X_2,
\cdots, X_N$ generate distributive lattices in Y and Z respectively.
Then the subspaces $W_1 \otimes Z, W_2 \otimes Z, \cdots, W_M
\otimes Z, Y \otimes X_1, Y \otimes X_2, \cdots, Y \otimes X_N$
generate a distributive lattice in $Y \otimes Z.$

\end{lemma}

\begin{proof}

By proposition \ref{PPP}, there exist index sets $\alpha, \beta$ and
decompositions $Y = \oplus_{i \in \alpha} T_i$ and $Z = \oplus_{i
\in \beta} U_i$ so that for every $i \in \{1, \cdots, M\}$ and $j
\in \{1, \cdots, N\}$ there are sets $\alpha_i \subseteq \alpha$ and
$\beta_j \subseteq \beta$ so that $W_i = \oplus_{k \in \alpha_i}
T_k$ and $X_j = \oplus_{k \in \beta_i} U_k.$ We can construct the
direct sum decomposition $Y \otimes Z = \oplus_{i \in \alpha, j \in
\beta} (T_i \otimes U_j).$ Then for any $i \in \{1, \cdots, M\}$,
$W_i \otimes Z = \oplus_{j \in \alpha_i, k \in \beta} (T_j \otimes
U_k)$ and $Y \otimes X_i = \oplus_{j \in \alpha, k \in \beta_i} (T_j
\otimes U_k)$.  Referring to proposition \ref{PPP} once more
completes the proof.

\end{proof}

In order to show a quadratic algebra is Koszul, we need to show that
for any positive integer $n$, the collection of subspaces
$V^iRV^{n-i-2}$ generates a distributive lattice in $V^n.$  We
repeat the following lemma from Serconek and Wilson's paper
\cite{justWS} which will allow us to reduce the problem to one where
we check for distributivity inside a much smaller vector space.

\begin{lemma} \label{graded}

Let $V = \sum_{i \in I} V_{[i]}$ be a graded vector space and $\{
X_j | j \in J \}$ be a collection of subspaces of $V$.  Assume that
each $X_j$ is graded, $X_j = \sum_{i \in I} X_{j,[i]},$ and
$X_{j,[i]} = X_j \cap V_{[i]}.$  Then $\{ X_j | j \in J \}$
generates a distributive lattice in $V$ if and only if, for every $i
\in I, \{X_{j,[i]} | j \in J \}$ generates a distributive lattice in
$V_{[i]}.$

\end{lemma}

\begin{lemma} \label{countdown}

Let $A$ be a quadratic algebra where the generators are partitioned
into the disjoint spaces $V_1, V_2, \cdots, V_N.$  Suppose every
relation is contained inside the space $V_{i+1}V_i$ for some $i \in
\{1, \cdots, n-1\}$. Consider the $\mathbb{Z}^n$ grading of $V^n$
where $x_1x_2 \cdots x_n$ is in $V^n_{[(z_1, z_2, \cdots, z_n)]}$ if
and only if $x_i \in V_i$ for all $i \in \{1, \cdots, n\}$.

The collection $\{ RV^{n-2}, \cdots, V^{n-2}R\}$ generates a
distributive lattice in $V^n$ if and only if $\{
RV^{n-2}_{[(n,n-1,\cdots, 1)]},$ $VRV^{n-3}_{[(n,n-1,\cdots, 1)]},$
$\cdots,$ $V^{n-2}R_{[(n,n-1,\cdots, 1)]} \}$ generates a
distributive lattice in $V^n_{[(n,n-1,\cdots, 1)]}.$

\end{lemma}

\begin{proof}

By lemma \ref{graded} we know that $\{ RV^{n-2}, \cdots, V^{n-2}R\}$
generates a distributive lattice in $V^n$ if and only if $\{
RV^{n-2}_{[(z_n,z_{n-1},\cdots, z_1)]},$
$VRV^{n-3}_{[(z_n,z_{n-1},\cdots, z_1)]},$ $\cdots,$
$V^{n-2}R_{[(z_n,z_{n-1},\cdots, z_1)]}$ generates a distributive
lattice in $V^n_{[(z_n,z_{n-1},\cdots, z_1)]} $ for every
$(z_n,z_{n-1},\cdots, z_1).$ Thus we only need to show that if $\{
RV^{n-2}_{[(n,n-1,\cdots, 1)]},$ $VRV^{n-3}_{[(n,n-1,\cdots, 1)]},$
$\cdots,$ $V^{n-2}R_{[(n,n-1,\cdots, 1)]} \}$ generates a
distributive lattice in $V^n_{[(n,n-1,\cdots, 1)]}$ then $\{
RV^{n-2}_{[(z_n,z_{n-1}, \cdots, z_1)]},$ $VRV^{n-3}_{[(z_n,z_{n-1},
\cdots, z_1)]},$ $\cdots,$ $V^{n-2}R_{[(z_n,z_{n-1}, \cdots, z_1)]}
\}$ generates a distributive lattice in $V^n_{[(z_n,z_{n-1}, \cdots,
z_1)]}$ for every $(z_n,z_{n-1},\cdots, z_1).$

Consider the collection $\{V_{z_n}V_{z_{n-1}}V_{z_{n-2}} \cdots
V_{z_{b+1}} W_i V_{z_{a-1}} V_{z_{a-2}}  \cdots V_{z_{1}} \}_{i \in
\{a+1, \cdots, b \} } $ where $W_i =$ $V_b V_{b-1}$ $\cdots$
$V_{i+1} R_i V_{i-2} V_{i-3}$ $\cdots$ $V_{a+1} V_a$.  It is enough
to show that this collection is distributive in $V^n_{[(z_n,
z_{n-1}, \cdots, z_2, z_1)]}$ where $z_{i+1} = z_{i} + 1$ for $i \in
\{a, \cdots, b-1\}$ because then we can string these increasing runs
together using lemma \ref{YZ} to complete the proof.

By our assumption, we do know that the collection $\{V_{n} V_{n-1}
V_{n-2} \cdots V_{b+1} W_i V_{a-1} V_{a-2} \cdots V_{1} \}_{i \in
\{a+1, \cdots, b \} }$ is distributive in $V^n_{[(n, n-1, \cdots, 2,
1)]}$ which implies the collection $\{ W_i \}_{i \in \{a+1, \cdots,
b \} } $ is distributive in $V^{b-a+1}_{[(b, b-1, \cdots, a+1,
a)]}$. This implies the distributivity of
$\{V_{z_n}V_{z_{n-1}}V_{z_{n-2}} \cdots V_{z_{b+1}} W_i V_{z_{a-1}}
V_{z_{a-2}} \cdots V_{z_{1}} \}_{i \in \{a+1, \cdots, b\} } $ in
$V^n_{[(z_n, z_{n-1}, \cdots, z_2, z_1)]}$ completing the proof.

\end{proof}

\begin{theorem} \label{pinch}

Suppose $\Gamma$ is a uniform layered graph on $V = \cup_{i=0}^N
V_i$ with exactly one vertex at layer $k$.  Let $\Gamma_0$ be the
induced subgraph of $\Gamma$ on $\cup_{i=0}^k V_i$ and $\Gamma_1$ be
the induced subgraph of $\Gamma$ on $\cup_{i=k}^N V_i$.    Then
$A(\Gamma)$ is Koszul if $A(\Gamma_0)$ and $A(\Gamma_1)$ are Koszul.

\end{theorem}

\begin{proof}
We work in terms of the associated graded algebra of $A(\Gamma)$
(equal to $B(\Gamma)^!$) so that the relations meet the hypotheses
of lemma \ref{countdown}.  The lemma allows us to assume the height
decreases at each step.

As $B(\Gamma_1)^!$ is Koszul we know that $$\{(\otimes_{i=j+1}^{N}
V_i) \otimes R_j \otimes (\otimes_{i=k+1}^{j-2} V_i) \}_{j \in \{
k+2, \cdots, N \}}$$ is distributive in $Y =  \otimes_{i=k+1}^{N}
V_i.$ As $B(\Gamma_0)^!$ is Koszul we know that
$$\{(\otimes_{i=j+1}^{k} V_i) \otimes R_j \otimes
(\otimes_{i=1}^{j-2} V_i) \}_{j \in \{ 2, \cdots, k \}}$$ is
distributive in $Z =  \otimes_{i=1}^{k} V_i.$

Applying lemma $\ref{YZ}$ gives us the distributivity of the
collection $$\{(\otimes_{i=j+1}^{N} V_i) \otimes R_j \otimes
(\otimes_{i=1}^{j-2} V_i) \}_{j \in \{ 2, \cdots, k \} \cup \{k+2,
\cdots, N \} }$$ in $\otimes_{i=1}^{N} V_i$.  Since $R_{k+1} = {0}$
this implies distributivity for $j \in \{ 2, \cdots, N \}$ thus
showing $B(\Gamma)^!$ and thus $A(\Gamma)$ is Koszul.

\end{proof}

We can now use this to reduce the number of cases we need to
consider.  Say that a layered graph $\Gamma$ is a \textit{$[z_0,
z_1, \cdots, z_N]$-graph} if $|V_i| = z_i$ for each $i$, where $V_i$
is the collection of vertices of layer $i.$

\begin{proposition} \label{no8}

Any uniform layered graph with unique minimal and maximal elements
and non-Koszul $A(\Gamma)$ has at least nine vertices.

\end{proposition}

\begin{proof}

Suppose the $[1,z_1,z_2, \cdots, z_{n-1}, 1]$-graph is minimal in
number of vertices amongst graphs with non-Koszul $A(\Gamma)$. Then
$z_i
> 1$ for all $0 < i < N.$ Otherwise we could use theorem \ref{pinch} to
find a smaller example.  We also know from \cite{cohom} that no
example of a non-Koszul $A(\Gamma)$ exists for graphs of height
three or less. This leaves one possibility for a non-Koszul
$A(\Gamma)$ with under nine vertices: that of a [1,2,2,2,1]-graph.

There are only ten such graphs with unique maximal and minimal
vertices \cite{sloane}, and of those only five \cite{unifsloane} are
uniform; see Figure \ref{fig:5unif12221}.
\begin{figure}[h!]
\centering%
\includegraphics[height=30mm]{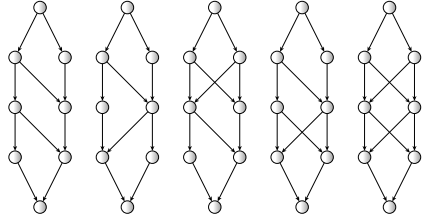}
\caption{The Five Uniform [1,2,2,2,1]-graphs} \label{fig:5unif12221}
\end{figure}
It is easy to check these five cases, by either hand or computer to
see these are all Koszul.

\end{proof}

\begin{theorem}

Consider the collection of all uniform layered graphs with unique
maximal and minimal elements.  $\Gamma = H$ is the minimal example
producing a non-Koszul $A(\Gamma).$

\end{theorem}

\begin{proof}

To clarify, by minimal we mean in terms of number of vertices though
it has been checked by computer that this is also the edge-minimal
example. Using proposition \ref{no8} we only need to show that $H$
is the unique example with nine vertices producing a non-Koszul
$A(\Gamma)$.

To show this, we need to check the ten uniform [1,3,2,2,1]-graphs,
ten uniform [1,2,2,3,1]-graphs, and twenty-three uniform
[1,2,3,2,1]-graphs (see Figures \ref{fig:10unif13221},
\ref{fig:10unif12231}, and \ref{fig:23unif12321}.)
\begin{figure}[h!]
\centering%
\includegraphics[height=25mm]{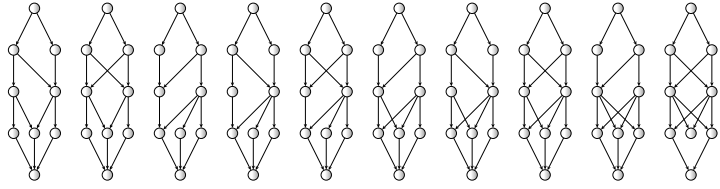}
\caption{The Ten Uniform [1,3,2,2,1]-graphs} \label{fig:10unif13221}
\end{figure}
\begin{figure}[h!]
\centering%
\includegraphics[height=25mm]{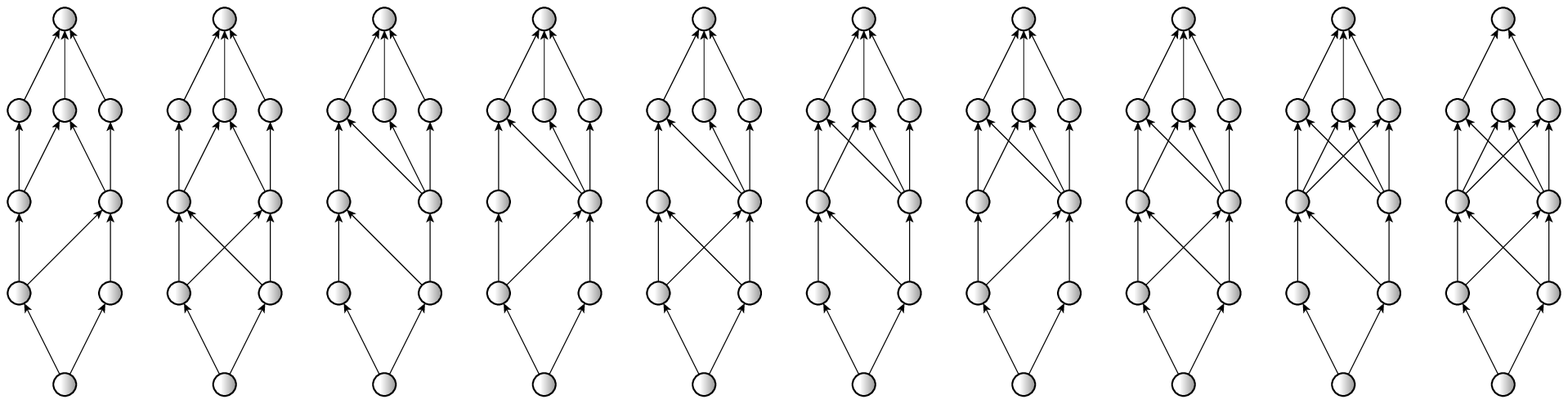}
\caption{The Ten Uniform [1,2,2,3,1]-graphs} \label{fig:10unif12231}
\end{figure}
\begin{figure}[h!]
\centering%
\includegraphics[height=50mm]{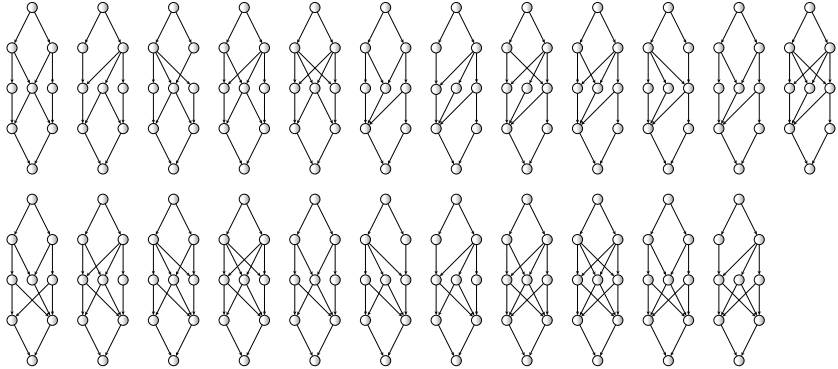}
\caption{The Twenty-three Uniform [1,2,3,2,1]-graphs}
\label{fig:23unif12321}
\end{figure}
This was done by computer, though these 43 cases individually can
still each be done by hand.

\end{proof}

The graph $H$ is minimal under more general conditions as well.  We
can first drop the requirement that the graph must have a unique
maximal element, but still require that maximal elements be of
maximal rank.  This changes very little, but does include the
possibility of a [1,2,2,2,2]-graph with non-Koszul $A(\Gamma)$.
There are thirty-five of these, twenty one of these being uniform,
also not too time consuming for us to check.

To generalize further, we drop the last condition and allow maximal
elements at every level.  Now that we include graphs like the one
shown in Figure \ref{fig:branchy},

\begin{figure}[h!]
\centering%
\includegraphics[height=30mm]{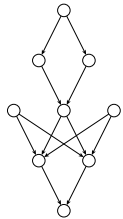}
\caption{An Arbitrary Uniform [1,2,3,2,1]-graph} \label{fig:branchy}
\end{figure}

\noindent we have many more examples to consider.  For ruling out
[1,2,2,2,1]-graphs, there are thirty-three cases needing to be
checked.  To show $H$ is minimal amongst graphs with nine vertices
we must check 83 [1,3,2,2,1]-graphs, 170 [1,2,3,2,1]-graphs, 93
[1,2,2,3,1]-graphs, and 65 [1,2,2,2,2]-graphs.  This is a bit much
to check by hand, but it can and has been done by computer.  With
this, the following has been shown:

\begin{theorem}

Consider the collection of all uniform layered graphs with a unique
minimal element.  $\Gamma = H$ is the minimal example producing a
non-Koszul $A(\Gamma).$

\end{theorem}

\section{Acknowledgements}

The authors would like to thank Robert Wilson for his guidance and
advice.  We would also like to thank programmer John Yeung for
introducing and mentoring us in Python, the language HiLGA was
written in, which led us to the discovery of the poset $H$.

\bibliographystyle{plain}
\bibliography{citations}

\end{document}